\documentclass[a4paper,final]{amsart}
\pdfoutput=1

\usepackage[utf8]{inputenc}
\usepackage{lmodern}
\usepackage[T1]{fontenc}

\usepackage{amsmath,amssymb,amsthm,enumerate}
\usepackage{mathpartir}
\usepackage{tikz}
\usetikzlibrary{arrows}
\usepackage{tikz-cd}

\usepackage{url}
\usepackage[hidelinks,urlcolor=black]{hyperref}

\theoremstyle{plain}
\newtheorem{theorem}{Theorem}
\newtheorem*{theorem*}{Theorem}
\newtheorem{lemma}[theorem]{Lemma}
\newtheorem*{lemma*}{Lemma}
\newtheorem{corollary}[theorem]{Corollary}
\newtheorem*{corollary*}{Corollary}

\newtheorem*{proposition*}{Proposition}

\newtheorem*{claim*}{Claim}
\newtheorem*{namedtheorem}{\theoremname}

\theoremstyle{definition}
\newtheorem{definition}[theorem]{Definition}
\newtheorem*{definition*}{Definition}

\newtheorem*{notation*}{Notation}

\newtheorem*{example*}{Example}

\newtheorem*{examples*}{Examples}

\theoremstyle{remark}
\newtheorem{remark}[theorem]{Remark}
\newtheorem*{remark*}{Remark}

\newtheorem*{note*}{Note}
\newtheorem*{acknowledgements*}{Acknowledgements}

\newcommand{\co}{\colon}

\newcommand{\case}{\emph{Case}}

\newcommand{\id}{\mathrm{id}}

\newcommand{\set}[1]{\{#1\}}

\renewcommand{\phi}{\varphi}
\renewcommand{\epsilon}{\varepsilon}

\newcommand{\us}{\vec u}
\newcommand{\vs}{\vec v}
\newcommand{\ws}{\vec w}

\newcommand{\ms}{\vec m}

\newcommand{\substpar}[1]{(#1)}
\newcommand{\esubst}[2]{#1 / #2}
\newcommand{\subst}[2]{\substpar{\esubst #1 #2}}

\newcommand{\II}{\mathbb{I}}

\newcommand{\emptyctx}{1} 
\newcommand{\pp}{\mathsf{p}}       
\newcommand{\qq}{\mathsf{q}}       

\newcommand{\CC}{\mathcal{C}}   
\newcommand{\op}{\mathrm{op}}
\newcommand{\Set}{\mathbf{Set}}

\DeclareMathOperator{\Hom}{Hom}

\renewcommand{\deg}{s} 

\newcommand{\indep}{\mathrel{\#}} 

\DeclareMathOperator{\yoneda}{\mathbf{y}}

\DeclareMathOperator{\Ty}{Ty}
\DeclareMathOperator{\Ter}{Ter}
\DeclareMathOperator{\Fill}{Fill}
\DeclareMathOperator{\KTy}{KTy}
\DeclareMathOperator{\AFib}{Contr}

\newcommand{\Path}{\mathsf{Path}}
\newcommand{\Pth}[1]{\mathsf{Path}_{#1}}

\newcommand{\nabs}[1]{\langle #1 \rangle}
\newcommand{\napp}{\mathop{\boldsymbol{@}}}

\newcommand{\ap}{\mathsf{ap}}

\newcommand{\Id}[1]{\mathsf{Id}_{#1}}
\DeclareMathOperator{\refl}{\mathsf{refl}}

\newcommand{\transp}{\mathsf{T_{\El}}} 
\newcommand{\trpeqv}{\mathsf{T^{\Equiv}_{\El}}} 

\newcommand{\fstr}{\kappa}  

\newcommand{\ext}{\mathsf{ext}} 

\newcommand{\Equiv}{\mathsf{Equiv}}

\newcommand{\app}{\mathsf{app}}

\DeclareMathOperator{\isContr}{\mathsf{isContr}}
\DeclareMathOperator{\isEquiv}{\mathsf{isEquiv}}
\newcommand{\fib}{\mathsf{fib}}
\DeclareMathOperator{\Gl}{\mathsf{G}} 
\DeclareMathOperator{\fGl}{\underline\Gl}

\newcommand{\ugl}{\mathsf{ug}}

\newcommand{\U}{\mathsf{U}}
\DeclareMathOperator{\El}{\mathsf{El}}
\DeclareMathOperator{\fEl}{\underline{\mathsf{El}}}
\newcommand{\code}[1]{\ulcorner #1 \urcorner}

\newcommand{\uaetop}{\mathsf{ua}}
\newcommand{\uabeta}{\mathsf{ua}_\beta}

\renewcommand{\i}{\mathsf{i}}

\title{The univalence axiom in cubical sets}
\date{\today}

\author{Marc Bezem}
\address{Department of Informatics\\
  University of Bergen\\
  Postboks 7800\\
  N-5020 Bergen\\
  Norway}
\email{bezem@ii.uib.no}%

\author{Thierry Coquand}
\address{Department of Computer Science and Engineering\\
  University of Gothenburg\\
  SE-412 96 Göteborg\\
  Sweden}
\email{thierry.coquand@cse.gu.se}

\author{Simon Huber}
\address{Department of Computer Science and Engineering\\
  University of Gothenburg\\
  SE-412 96 Göteborg\\
  Sweden}
\email{simon.huber@cse.gu.se}

\begin{document}

\begin{abstract}
  In this note we show that Voevodsky's univalence axiom holds in the
  model of type theory based on cubical sets as described
  in~\cite{BezemCoquandHuber14,Huber15}.  We will also discuss Swan's
  construction of the identity type in this variation of cubical sets.
  This proves that we have a model of type theory supporting dependent
  products, dependent sums, univalent universes, and identity types
  with the usual judgmental equality, and this model is formulated in
  a constructive metatheory.
\end{abstract}

\maketitle

\section{Review of the cubical set model}
\label{sec:review}

We give a brief overview of the cubical set model, introducing some
different notations, but will otherwise assume the reader is familiar
with~\cite{BezemCoquandHuber14,Huber15}.

As opposed to~\cite{BezemCoquandHuber14,Huber15} let us define cubical
sets as contravariant presheaves on the opposite of the category used
there, that is, the category of cubes $\CC$ contains as objects finite
sets $I = \set{i_1,\dots,i_n}$ ($n \ge 0$) of names and a morphism $f
\co J \to I$ is given by a set-theoretic map $I \to J \cup \set{0,1}$
which is injective when restricted to the preimage of $J$; we will
write compositions in applicative order.  The category of cubical sets
is the category $[\CC^\op,\Set]$ of presheaves on $\CC$.  A morphism
$f \co J \to I$ in $\CC$ can be viewed as a substitution.  If $f(i)
\in J$, we call $f$ \emph{defined on} $i$.  For $i \notin I$, the face
morphisms are denoted by $\subst i 0, \subst i 1 \co I \to I,i$ and
are induced by setting $i$ to $0$ and $1$, respectively; degenerating
along $i \notin I$ is denoted by $\deg_i \co I,i \to I$ and is induced
by the inclusion $I \subseteq I,i$.

If $\Gamma$ is a cubical set, we write $\Ty (\Gamma)$ for the
collection/class of presheaves on the category of elements of
$\Gamma$~\cite{BezemCoquandHuber14,Huber15}.  Such a presheaf $A \in
\Ty (\Gamma)$ is given by a family of sets $A (I,\rho)$ for $I \in
\CC$ and $\rho \in \Gamma(I)$ together with restriction functions.  As
$\rho \in \Gamma(I)$ determines $I$ we simply write $A \rho$ for $A
(I,\rho)$.  Given $A \in \Ty(\Gamma)$ and a natural transformation
(substitution) $\sigma \co \Delta \to \Gamma$ we get $A \sigma \in \Ty
(\Delta)$ defined as $(A \sigma) \rho = A (\sigma \rho)$ which
extends canonically to the restrictions.
For $A \in \Ty (\Gamma)$ we denote the set of sections of $A$ by
$\Ter(\Gamma, A)$; so $a \in \Ter(\Gamma,A)$ is given by a family $a
\rho \in A \rho$ for $\rho \in \Gamma (I)$ such that $(a \rho) f = a
(\rho f)$ for $f \co J \to I$.  Substitution also extends to terms via
$(a\sigma) \rho = a (\sigma \rho)$.

Let us recall the construction of $\Pi$-types: $\Pi\,A\,B \in
\Ty(\Gamma)$ for $A \in \Ty(\Gamma)$ and $B \in \Ty(\Gamma.A)$ is
given by letting each element $w$ of $(\Pi\,A\,B)\rho$ (with $\rho \in
\Gamma(I)$) be a family of $w_f\,a \in B (\rho f, a)$ for $f \co J \to
I$ and $a \in A\rho$ satisfying $(w_f\,a)g = w_{fg}\,(a g)$; the
restriction of such a $w$ is given by $(wf)_g = w_{fg}$.  In the
sequel we will however only have to refer to $w_f$ when $f$ is the
identity, and will thus simply write $w\,a$ for $w_\id\,a$.  We also
occasionally switch between sections in $\Ter(\Gamma.A, B)$ and
$\Ter(\Gamma,\Pi A B)$ without warning the reader.

Let $A \in \Ty (\Gamma)$, $\rho \in \Gamma (I)$, and $J \subseteq I$.
A \emph{$J$-tube} in $A$ over $\rho$ is given by a family $\us$ of
elements $u_{jc} \in A \rho \subst j c$ for $(j,c) \in J \times
\set{0,1}$ which is adjacent compatible, that is, $u_{jc} \subst k d =
u_{k d} \subst j c$ for $(j,c),(k,d) \in J \times \set{0,1}$.  For
$(i,a) \in (I - J) \times \set{0,1}$ we say that an element $u_{ia}
\in A \rho \subst i a$ is a \emph{lid} of such a tube $\us$ if $u_{jc}
\subst i a = u_{ia} \subst j c$ for all $(j,c) \in J \times
\set{0,1}$.  In this situation we call the pair $[J \mapsto \us; (i,a)
\mapsto u_{ia}]$ an \emph{open box} in $A$ over $\rho$.  A
\emph{filler} for such an open box is an element $u \in A \rho$ such
that $u \subst j c = u_{jc}$ for $(j,c) \in \set{(i,a)} \cup (J \times
\set{0,1})$.  In case $J$ is empty, we simply write $[(i,a) \mapsto
u_{ia}]$.

Given $f \co K \to I$ and an open box $m = [J \mapsto \us; (i,a)
\mapsto u_{ia}]$ in $A$ over $\rho$ we call $f$ allowed for $m$ if $f$
is defined on $J,i$.  In this case we define the open box $mf$ in $A$
in $\rho f$ to be $[J f \mapsto \us f; (f (i), a) \mapsto u_{ia} (f -
i)]$ where $\us f$ is given by $(\us f)_{f(j)\, c} = u_{j c} (f - j)$
with $f - i \co K - {f (i)} \to I-i$ being like $f$ but skipping $i$,
and $J f$ is the image of $J$ under $f$.

Recall from~\cite[Section~4]{BezemCoquandHuber14} that a
\emph{(uniform) Kan structure} for a type $A \in \Ty(\Gamma)$ is given
by an operation $\fstr$ which (uniformly) fills open boxes: for any
$\rho \in \Gamma(I)$ and open box $m$ in $A$ over $\rho$ we get a filler
$\fstr\,\rho\,m$ of $m$ subject to the uniformity condition
\[
(\fstr\,\rho\,m) f = \fstr\,(\rho f)\,(m f)
\]
for all $f \co K \to I$ allowed for $m$.

Any Kan structure $\fstr$ defines a \emph{composition operation}
$\bar\fstr$ which provides the missing lid of the open box, given by:
\begin{align*}
  \bar\fstr\,\rho\,[J \mapsto \us; (i,0) \mapsto u_{i0}]
  &= (\fstr\,\rho\,[J \mapsto \us; (i,0) \mapsto u_{i0}]) \subst i 1
  \\
  \bar\fstr\,\rho\,[J \mapsto \us; (i,1) \mapsto u_{i1}]
  &= (\fstr\,\rho\,[J \mapsto \us; (i,1) \mapsto u_{i1}]) \subst i 0
\end{align*}

We denote the set of all Kan structures on $A \in \Ty (\Gamma)$ as
$\Fill (\Gamma, A)$.  If $\sigma \co \Delta \to \Gamma$ and $\fstr$ is
an element in $\Fill (\Gamma, A)$, we get an element $\fstr \sigma$ in
$\Fill (\Delta, A \sigma)$ defined by $(\fstr \sigma) \, \rho =
\fstr \, (\sigma \rho)$.

Given a cubical set $\Gamma$ a \emph{Kan type} is a pair $(A,\fstr)$
where $A \in \Ty (\Gamma)$ and $\fstr \in \Fill (\Gamma,A)$.  We
denote the collection of all such Kan types by $\KTy (\Gamma)$.
In~\cite{BezemCoquandHuber14} we showed that Kan types are closed
under dependent products and sums constituting a model of type theory.

\section{Path types}\label{sec:paths}

In~\cite{BezemCoquandHuber14} we introduced identity types which were
however only ``weak'', e.g., transport along reflexivity is only
propositionally equal to the identity function but not necessarily
judgmentally equal.  For this reason we will call these types
\emph{path types} and reserve $\Id{A}$ for the identity type with the
usual judgmental equality defined in Section~\ref{sec:Id}.

Recall that the path type $\Pth{A}\,u\,v \in \Ty(\Gamma)$ for $A \in
\Ty (\Gamma)$ and $u,v \in \Ter (\Gamma,A)$ is defined by the sets
$(\Pth{A}\,u\,v)\rho$ containing equivalence classes $\nabs i w$ where
$i \notin I$ and $w\in A\rho\deg_i$ with $w\subst i 0 = u \rho$ and
$w\subst i 1 = v \rho$.  Restrictions are defined as expected, and we
showed that Kan types are closed under forming path
types~\cite{BezemCoquandHuber14}.

It will be convenient below to introduce paths using separated
products.

\begin{definition}\label{def:separation}
  Given cubical sets $\Gamma$ and $\Delta$, we say that
  $u\in\Gamma(I)$ and $v\in\Delta(I)$ are \emph{separated}, denoted by
  $u \indep v$, if they come through degeneration from cubes with
  disjoint sets of directions. More precisely, if there are $J
  \subseteq I$, $K \subseteq I$ with $J \cap K = \emptyset$ and $u'
  \in \Gamma (J), v' \in \Delta (K)$ such that $u = u' \deg$ and $v =
  v' \deg'$ with $\deg$ and $\deg'$ induced by the inclusion $J
  \subseteq I$ and $K \subseteq I$, respectively.

  The \emph{separated product} $\Gamma * \Delta$ of $\Gamma$ and
  $\Delta$ is the cubical set defined by
  \[
    (\Gamma * \Delta) (I) =%
    \set{ (u,v) \in \Gamma (I) \times \Delta (I)\mid u \indep v }
    \subseteq (\Gamma \times \Delta) (I).
  \]
  The restrictions are inherited from $\Gamma \times \Delta$, that is,
  they are defined component wise.  It can be shown that ${-} * {-}$
  extends to a functor, and that ${-} * \Delta$ has a right
  adjoint.  
\end{definition}

Of particular interest is $\Gamma*\II$ where $\II$ is the interval
defined by $\II (J) = J \cup \set {0,1}$
(see~\cite[Section~6.1]{BezemCoquandHuber14}). Then
\[
(\Gamma * \II) (I) =%
(\Gamma (I) \times \set {0,1}) \cup%
\set { (\rho \deg_i, i) \mid i \in I \land \rho \in \Gamma (I-i)}.
\]
If $(\rho, i) \in (\Gamma * \II) (I)$ with $i\in I$, then $\rho =
\rho' \deg_i$ for a uniquely determined $\rho'$ which we denote by
$\rho - i$.

We can use $\Gamma*\II$ to formulate the following introduction
rule for path types
\begin{mathpar}
  \inferrule %
  { A \in \Ty (\Gamma) \\
    w \in \Ter (\Gamma * \II, A \pp) } %
  { \nabs {} w \in \Ter(\Gamma, \Pth{A}\,w[0]\,w[1])}%
\end{mathpar}
where $[0],[1] \co \Gamma \to \Gamma * \II$ are induced by the global
elements $0$ and $1$ of $\II$, respectively, and $\pp \co \Gamma*\II
\to \Gamma$ is the first projection.  The binding operation is
interpreted by $(\nabs {} w) \rho = \nabs i \, w (\rho \deg_i, i)$
with $i$ a fresh name (see~\cite[Section~8.2]{BezemCoquandHuber14}).

Given an element $\nabs i w \in (\Pth{A}\,u\,v) \rho$ with $\rho \in
\Gamma (I)$, we set $(\nabs i w) \napp a = w \subst i a$ where $a$ is
$0$, $1$, or a fresh name.

\section{Equivalences and univalence}\label{sec:equis}

We will now recall the definition of an equivalence as a map having
contractible fibers and then derive an operation for contractible and
Kan types.  To enhance readability we define the following types using
variable names:
\begin{align*}
  \isContr A &= \Sigma (x : A)\, \Pi (y : A) \, \Pth{A}\,x\,y
  \\
  \fib\,t\,v &= \Sigma (x : A) \, \Pth{B}\,(t \, x)\,v
  \\
  \isEquiv t &= \Pi (y : B) \, \isContr (\fib\,t\,y)
  \\
  \Equiv\,A\,B &= \Sigma (t : A \to B) \, \isEquiv t
\end{align*}
where $A$ and $B$ are types, $t : A \to B$, and $v : B$ (all in an
ambient context $\Gamma$).  This can of course also be formally
written name-free: for example, the first type can be written as
$\Sigma A \Pi A\pp \, (\Pth{A\pp\pp} \qq\pp ~\qq) \in \Ty(\Gamma)$ and
the second one as $\Sigma A\pp\, (\Pth{B\pp\pp} \,\app
(t\pp\pp,\qq)\,\qq\pp) \in \Ty(\Gamma.B)$.

\begin{definition}
  \label{def:acyclic-fibration}
  A \emph{(uniform) acyclic-fibration structure} on a type $A \in \Ty
  (\Gamma)$ is given by an operation $\ext$ uniformly filling any
  tube, that is, given $\rho \in \Gamma (I)$, $J \subseteq I$, a
  $J$-tube $\us$ in $A \rho$, we have
  \[
  \ext \,\rho\,[J\mapsto {\us}] \in A \rho
  \]
  extending $\us$ (so $(\ext\,\rho\,[J\mapsto\us]) \subst i a =
  u_{ia}$ for $(i,a) \in J \times \set{0,1}$) and for $f \co K \to I$
  defined on $J$ we have
  \[
  (\ext\,\rho\,[J\mapsto {\us}]) f = \ext\,(\rho f)\,[J f \mapsto \us f].
  \]
  We denote the set of acyclic-fibration structures on $A \in
  \Ty(\Gamma)$ by $\AFib(\Gamma,A)$.
\end{definition}

Note that given $\ext \in \AFib (\Gamma,A)$ and $\sigma \co \Delta \to
\Gamma$ we obtain $\ext\, \sigma \in \AFib (\Delta, A\sigma)$ via
$(\ext \,\sigma)\, \rho = \ext \, (\sigma \rho)$.

\begin{lemma}
  \label{lem:contractible}
  Given a type $A$ in $\Ty(\Gamma)$ we have maps
  \[
    \begin{tikzcd}[column sep=2cm]
      \Fill(\Gamma,A) \times \Ter(\Gamma, \isContr A) %
      \arrow[r,"\phi",shift left] %
      &
      \arrow[l,"{\langle \psi_0,\psi_1 \rangle}",shift left]
      \AFib(\Gamma,A),
    \end{tikzcd}
  \]
  with $\phi \, {\langle \psi_0,\psi_1 \rangle} = \id$.  Moreover, these
  maps are natural: if $\sigma \co \Delta \to \Gamma$, then
  $(\phi\,\fstr\,p) \sigma = \phi\,(\fstr \sigma)\,(p \sigma)$,
  $(\psi_0\,\ext) \sigma = \psi_0\,(\ext\,\sigma)$, and $(\psi_1\,\ext)
  \sigma = \psi_1\,(\ext\,\sigma)$.
\end{lemma}
\begin{proof}
  Let $\fstr \in \Fill(\Gamma,A)$ and $p \in \Ter (\Gamma, \isContr
  A)$. To define $\phi\,\fstr\,p \in \AFib(\Gamma,A)$, let $\rho \in
  \Gamma (I)$ and $\us$ a $J$-tube in $A$ over $\rho$.  We take a
  fresh dimension $i$ and form an open box with the center of
  contraction $p\rho.1$ at the closed end and $\us$ at the open end,
  connected by $p\rho.2$; filling this gives us an extension of $\us$.
  Formally:
  \[
    \phi\,\fstr\,p\,\rho\,[J\mapsto {\us}] =%
    \bar\fstr\,(\rho \deg_i)\,[J \mapsto (p \rho.2 \, \us) \napp i;
    (i,0) \mapsto p\rho.1]
  \]
  where $(p\rho.2 \, \us) \napp i$ is the $J$-tube given by $(p(\rho
  \subst j c).2\, u_{jc}) \napp i$ at side $(j,c) \in J \times
  \set{0,1}$.

  Conversely, let $\ext \in \AFib(\Gamma,A)$. To get a Kan structure
  we first fill the missing lid and then the interior, that is, we set
  \begin{multline*}
    \psi_0\,\ext\,\rho\,[J \mapsto \us; (i,0) \mapsto u_{i0}] =\\
    \ext\,\rho\,[J\mapsto \us, (i,0) \mapsto u_{i0}, (i,1) \mapsto
    \ext\,\rho\subst i 1\,[J\mapsto \us \subst i 1]],
  \end{multline*}
  and likewise for the other filling.  To define $\psi_1\,\ext\,\rho$
  we choose $\ext\,\rho\,[]$ as center of contraction, which is
  connected to any $a \in A \rho$ by the path
  \[
    \nabs i \, \ext\,(\rho \deg_i)\,[(i,0) \mapsto \ext\,\rho\,[],
    (i,1) \mapsto a].
  \]

  One can show uniformity, naturality, and that $\langle \psi_0,\psi_1
  \rangle$ is a section of $\phi$.
\end{proof}

Next we will define an operation $\Gl$ which allows us to transform an
equivalence into a ``path''\footnote{We will see later that this
  indeed induces a path in a universe whenever both types $A$ and $B$
  are small.}.  This operation was introduced
in~\cite{Coquand13} and motivated the ``glueing'' operation
of~\cite{CohenCoquandHuberMortberg15}.  We will define it in such a
way that the associated transport of this path is given by underlying
map of the equivalence.

A useful analogy is provided by the notion of \emph{pathover}, a
heterogeneous path lying over another path.  We shortly review this
notion from type theory with inductive equality.  Given a type family
$P\co T\to\U$ and a path $p\co x=_T y$ with its transport function
$p_*\co Px\to Py$. If $Px$ and $Py$ are different types, there is no
ordinary path connecting $u \co Px$ and $v \co Py$. Therefore the
pathovers connecting $u$ and $v$ are taken to be the paths of type
$p_* u =_{Py} v$ (in the fiber $Py$).

We apply the same idea to $\Gl t$, which should be a path from $A$ to
$B$ in $\U$ such that transport along this path is $t \co A\to B$.
For the type family $P$ we take $\id_\U$ such that $A$ and $B$ indeed
are fibers of $P$. Intuitively, a path from $A$ to $B$ is a set of
heterogeneous paths between elements $a$ of $A$ and $b$ of $B$.  We
want $t$ to be the transport function along the path from $A$ to $B$.
By analogy we would take $\Gl t$ to be the set of pathovers connecting
$a\co A$ and $b\co B$ defined as the set of paths in $B$ connecting
$t~a$ and $b$. However, since we must be able to recover the
startpoint $a$, we define $\Gl t$ to be the set of pairs consisting of
$a\co A$ and a path connecting $t~a$ and $b$. (Unlike $a$, the
endpoint $b$ can be recovered from the pathover and need not be
remembered.)

With the above informal explanation in mind, we define the operation
$\Gl$ first on cubical sets and then explain how it lifts to Kan
structures.  It satisfies the rules:
\begin{equation}
  \label{eq:glue-form}
  \inferrule %
  { A \in \Ty (\Gamma) \\
    B \in \Ty (\Gamma) \\
    t \in \Ter(\Gamma, A \to B) }%
  { \Gl t \in \Ty(\Gamma * \II)}%
  \qquad %
  \inferrule {}
  {  (\Gl t) [0] = A \in \Ty(\Gamma) \\\\
     (\Gl t) [1] = B \in \Ty(\Gamma)} %
\end{equation}
\begin{equation}
  \label{eq:glue-subst}
  \inferrule %
  { \sigma \co \Delta \to \Gamma \\
    A \in \Ty (\Gamma) \\
    B \in \Ty (\Gamma) \\
    t \in \Ter(\Gamma, A \to B)
  }%
  { (\Gl t) (\sigma * \id) = \Gl (t \sigma) \in \Ty(\Delta *
    \II) }%
\end{equation}
The latter rule expresses stability under substitutions.  Here and
below $\Gl$ (and $\ugl$ below) have $A$ and $B$ as implicit arguments.

\begin{definition}\label{def:Gluetype}
  Assume the premiss of \eqref{eq:glue-form} and define for every
  $\rho \in \Gamma (I)$:
  \begin{equation}
    \label{eq:glue-def}
    \begin{split}
      (\Gl t) (\rho, 0) &= A \rho, \text{ with restrictions as in
        $A$,}\\
      (\Gl t) (\rho, 1) &= B \rho, \text{ with restrictions as in $B$,
        and }\\
      (\Gl t) (\rho, i) &= \set { (u,v) \mid u \in A (\rho-i) \land v
        \in B \rho \land v \subst i 0 = t (\rho-i)\, u}.
    \end{split}
  \end{equation}
  In the last case $\rho \indep i$, so $\rho = (\rho - i) \deg_i$.
  The restrictions in the latter case are a little involved.  We need
  $(u,v)f \in (\Gl t) (\rho f, f(i))$ for $f \co J \to I$.  If
  $f(i)=0$, we take $(u,v)f = u \deg_i f$, indeed in $A\rho f$.  If
  $f(i)=1$, we take $(u,v)f = vf$, indeed in $B\rho f$.  Finally, if
  $f$ is defined on $i$, we have $f - i \co J - {f (i)} \to I-i$ and
  we define $(u,v) f = (u (f-i), v f)$, which is indeed correct as
  $(\rho-i)(f-i) = \rho f -f(i)$ under the given assumptions.  It can
  then be checked that the restrictions satisfy the presheaf
  requirements.  This concludes the definition of $\Gl t$.
\end{definition}

We have a map $\ugl \in \Ter (\Gamma * \II.\Gl t, B \pp)$ given by:
\begin{mathpar}
  \ugl ((\rho, 0), u) = t \rho\, u
  \and %
  \ugl ((\rho,1),v) = v
  \and %
  \ugl ((\rho,i),(u,v)) = v
\end{mathpar}

The fact that a map $t \in \Ter (\Gamma,A \to B)$ is an equivalence
can be represented as an element of $\AFib(\Gamma.B, \fib\,t)$.  By
Lemma~\ref{lem:contractible} this is the case whenever $A$ and $B$
have Kan structures and the fibers of $t$ are contractible.

\begin{theorem}
  \label{thm:glue-kan}
  The operation $\Gl$ can be lifted to Kan structures provided $t$ is
  an equivalence, i.e., there is an operation $\fGl$ which given the
  premiss of \eqref{eq:glue-form} and $\fstr_A \in \Fill (\Gamma,A)$,
  $\fstr_B \in \Fill (\Gamma,B)$, and $\ext \in \AFib(\Gamma.B,
  \fib\,t)$ returns $\fGl\,\fstr_A\,\fstr_B\,\ext \in
  \Fill(\Gamma*\II,\Gl t)$.  This operation satisfies
  \begin{align*}
    (\fGl\,\fstr_A\,\fstr_B\,\ext) [0] &= \fstr_A\\
    (\fGl\,\fstr_A\,\fstr_B\,\ext) [1] &= \fstr_B\\
    (\fGl\,\fstr_A\,\fstr_B\,\ext) (\sigma * \id) &= \fGl\,(\fstr_A
    \sigma)\,(\fstr_B \sigma)\,(\ext \,(\sigma \pp,\qq))
  \end{align*}
  where $\sigma \co \Delta \to \Gamma$.
\end{theorem}

\begin{proof}
  To define $(\fGl\,\fstr_A\,\fstr_B\,\ext) (\rho,r)$ for $(\rho, r)
  \in (\Gamma * \II) (I)$ we argue by cases.  For $r = 0,1$ we take:
  \begin{align*}
    (\fGl\,\fstr_A\,\fstr_B\,\ext) (\rho,0) &= \fstr_A \, \rho\\
    (\fGl\,\fstr_A\,\fstr_B\,\ext) (\rho,1) &= \fstr_B \, \rho
  \end{align*}
  Let us now consider the main case where $r = i \in I$ is a name and
  thus $\rho \indep i$, $\rho = (\rho - i) \deg_i$.  We are given $j$
  (the name along which we fill), $\ws$ a $J$-tube in $(\Gl t)$ over
  $(\rho,i)$ (with $J \subseteq I - j$), and $w_{ja} \in (\Gl t)
  (\rho,i) \subst j a$ for $a = 0$ or $1$, which fits $\vec w$.  We
  want to define
  \[
  w := (\fGl\,\fstr_A\,\fstr_B\,\ext)\, (\rho,i)\,[ J \mapsto \ws;
  (j,a) \mapsto w_{ja}]
  \]
  in $(\Gl t) (\rho,i)$.  For this we have to construct $w = (u,v)$
  with $u \in A (\rho-i)$ and $v \in B\rho$ such that $v \subst i
  0 = t (\rho - i)\, u$.

  We can map $w_{ja},\ws$ using $\ugl$ and obtain an open box
  $v_{ja},\vs$ in $B$ over $\rho$ given by
  \[
  v_{kb} := \ugl ((\rho, i) \subst k b, w_{kb}) \in B \rho
  \subst k b.
  \]
  There are four cases to consider depending on how the open box
  relates to the direction $i$.  Each case will be illustrated
  afterwards with simplified $J$.  Note that in all these pictures the
  part in $A$ is mapped by $t$ to the left face of the part in $B$.
  Here are the four cases:
  \medskip

  \case\ $i \neq j$ and $i \notin J$. We extend the $J$-tube $\ws$ to
  $J,i$-tube by constructing $w_{i0}$ and $w_{i1}$ and then proceed as
  in the next case with the tube $\ws, w_{i0}, w_{i1}$.  Note that we
  want
  \begin{align*}
    w_{i0} &\in (\Gl t) (\rho,i) \subst i 0 = A (\rho - i), \text {
             and}
    \\
    w_{i1} &\in (\Gl t) (\rho,i) \subst i 1 = B (\rho - i),
  \end{align*}
  so we can take
  \begin{align}
    \label{eq:wlog-case-a}
    w_{i0}
    &= \fstr_A\,(\rho-i)\,[J \mapsto \ws \subst i 0; (j,a) \mapsto w_{ja}
       \subst i 0 ], \text{ and}
    \\
    \label{eq:wlog-case-b}
    w_{i1}
    & = \fstr_B\,(\rho-i)\,[J \mapsto \ws \subst i 1; (j,a) \mapsto w_{ja}
      \subst i 1 ].
  \end{align}
  The resulting open box is compatible by construction.  Note that
  this (together with the cases for $r=0$ and $r=1$) also ensures that
  the Kan structure satisfies the equations in~\eqref{eq:glue-form}.

  We illustrate this case in the picture below.  Here and below the
  left part is in $A$ and on the right we have the open box $\vec v$
  in $B$.  For simplicity we also omit $\rho$.  We construct $w_{i0}$
  and $w_{i1}$ by filling the open boxes indicated by thicker lines on
  the left and on the right, respectively.
  
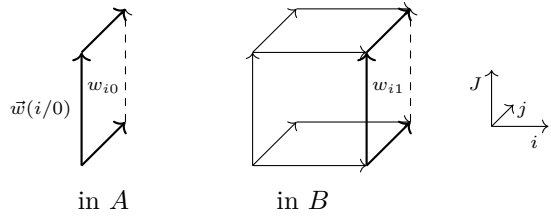
\begin{equation*}
    \begin{tikzpicture}
      [scale=1.5,
      baseline={([yshift=-.5ex]
        current bounding box.center)}] 
      \node at (0,-0.5,0.5)  {in $A$};
      \node at (1.75,-0.5,0.5) {in $B$};
      \coordinate (A00) at (0,0,0);
      \coordinate (A01) at (0,0,1);
      \coordinate (A10) at (0,1,0);
      \coordinate (A11) at (0,1,1);
      \draw [->,thick] (A01) --
      node [left] {\scriptsize$\ws\subst i 0$} (A11) ;
      \draw [->,thick] (A11) -- (A10) ;
      \draw [->,thick] (A01) -- (A00) ;
      \draw [->,dashed,thin] (A00) -- (A10) ;
      \coordinate (BL00) at (1.5,0,0);
      \coordinate (BL01) at (1.5,0,1);
      \coordinate (BL10) at (1.5,1,0);
      \coordinate (BL11) at (1.5,1,1);
      \coordinate (BR00) at (2.5,0,0);
      \coordinate (BR01) at (2.5,0,1);
      \coordinate (BR10) at (2.5,1,0);
      \coordinate (BR11) at (2.5,1,1);
      \draw [->] (BL01) -- (BL11) ;
      \draw [->] (BL11) -- (BL10) ;
      \draw [->] (BL01) -- (BL00) ;
      \draw [->,thick] (BR01) -- (BR11) ;
      \draw [->,thick] (BR11) -- (BR10) ;
      \draw [->,thick] (BR01) -- (BR00) ;
      \draw [->,dashed,thin] (BR00) -- (BR10) ;
      \draw [->] (BL00) -- (BR00) ;
      \draw [->] (BL01) -- (BR01) ;
      \draw [->] (BL10) -- (BR10) ;
      \draw [->] (BL11) -- (BR11) ;
      \node at (0,0.5,0.5)  {\scriptsize$w_{i0}$};
      \node at (2.5,0.5,0.5)  {\scriptsize$w_{i1}$};
      \draw [->] (3.5,0.25,0.75) -- node [left,near end]
      {\scriptsize$J$} (3.5,0.75,0.75);
      \draw [->] (3.5,0.25,0.75) -- node [right,near end]
      {\scriptsize$j$} (3.5,0.25,0.25);
      \draw [->] (3.5,0.25,0.75) -- node [below,near end]
      {\scriptsize$i$} (4,0.25,0.75);
    \end{tikzpicture}
  \end{equation*}

  \case\ $i \neq j$ and $i \in J$. In this case $v_{i0} =
  \ugl((\rho{\subst i 0},0),w_{i0}) = t \rho{\subst i 0} \, w_{i0} =
  t (\rho-i) \,w_{i0}$ since $\rho\indep i$.  We can therefore take $w =
  (w_{i0}, v) \in (\Gl t) (\rho,i)$ where $v = \fstr_B\,(\rho-i)\,[J
  \mapsto \vs; (j,a) \mapsto v_{ja}]$.  This can be illustrated by:
  
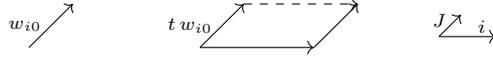
\[
  \begin{tikzpicture}[scale=1.5]
    \draw [->] (0,0,1) -- node [left] {\scriptsize$w_{i0}$} (0,0,0) ;
    \draw [->] (1.5,0,1) --
    node [left] {\scriptsize$t\,w_{i0}$} (1.5,0,0) ;
    \draw [->,dashed] (1.5,0,0) -- (2.5,0,0) ;
    \draw [->] (1.5,0,1) -- (2.5,0,1) ;
    \draw [->] (2.5,0,1) -- (2.5,0,0) ;
    \draw [->] (3.5,0,0.75) -- node [left,near end]
    {\scriptsize$J$} (3.5,0,0.25);
    \draw [->] (3.5,0,0.75) -- node [above=-0.5ex,near end]
    {\scriptsize$i$} (4,0,0.75);
  \end{tikzpicture}
  \]

  \case\ $j = i$ and $a = 0$.  Like in the previous case we can take
  $w = (w_{i0}, v) \in (\Gl w) (\rho,i)$ where $v =
  \fstr_B\,(\rho-i)\,[J \mapsto \vs; (j,a) \mapsto v_{ja} ]$.
  This case is illustrated as follows:
  
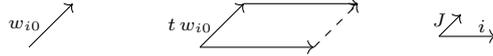
\[
  \begin{tikzpicture}[scale=1.5]
    \draw [->] (0,0,1) -- node [left] {\scriptsize$w_{i0}$} (0,0,0) ;
    \draw [->] (1.5,0,1) --
    node [left] {\scriptsize$t\,w_{i0}$} (1.5,0,0) ;
    \draw [->] (1.5,0,0) -- (2.5,0,0) ;
    \draw [->] (1.5,0,1) -- (2.5,0,1) ;
    \draw [->,dashed] (2.5,0,1) -- (2.5,0,0) ;
    \draw [->] (3.5,0,0.75) -- node [left,near end]
    {\scriptsize$J$} (3.5,0,0.25);
    \draw [->] (3.5,0,0.75) -- node [above=-0.5ex,near end]
    {\scriptsize$i$} (4,0,0.75);
  \end{tikzpicture}
  \]

  \case\ $j = i$ and $a = 1$. In this case the direction of the
  filling is opposite to $t$, and therefore we have to use $\ext$
  which expresses that $\fib\,t$ is contractible.
  The family $\ms$ defined by
  \[
  m_{kb} := (w_{kb}, \nabs i \, v_{kb}) \in (\fib\,t) ((\rho-i)
  \subst k b, w_{i1} \subst k b)
  \]
  for $(k,b) \in J \times \set{0,1}$ constitutes a $J$-tube over
  $(\rho-i,w_{i1})$ in the contractible type $\fib\,t \in
  \Ty(\Gamma.B)$.

  So we can extend this tube to obtain
  \[
    (u,\omega) = \ext\,(\rho-i,w_{i1})\,[J\mapsto\ms] \in
    (\fib\,t)(\rho-i,w_{i1})
  \]
  and we can take $w := (u,\omega \napp i) \in (\Gl t) (\rho,i)$.

  Let us illustrate this case: we are given the two dots on the left
  and the solid lines on the right in the picture below, and we want
  to construct the dashed line and a square on the right such that the
  dashed line is mapped to the dotted line via $t$, that is, we
  basically want to construct an element in the fiber of $w_{i1}$
  under $t$.
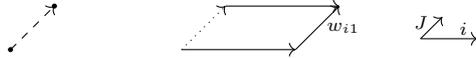
  \[
  \begin{tikzpicture}[scale=1.5]
    \draw [->,dashed] (0,0,1) -- (0,0,0) ;
    \filldraw [solid]
    (0,0,1) circle (0.5pt)
    (0,0,0) circle (0.5pt);
    \draw [->,dotted] (1.5,0,1) -- (1.5,0,0) ;
    \draw [->] (1.5,0,0) -- (2.5,0,0) ;
    \draw [->] (1.5,0,1) -- (2.5,0,1) ;
    \draw [->] (2.5,0,1) -- node [right] {\scriptsize$w_{i1}$} (2.5,0,0) ;
    \draw [->] (3.5,0,0.75) -- node [left,near end]
    {\scriptsize$J$} (3.5,0,0.25);
    \draw [->] (3.5,0,0.75) -- node [above=-0.5ex,near end]
    {\scriptsize$i$} (4,0,0.75);
  \end{tikzpicture}
  \]
  This concludes the definition of the filling operations of $\Gl t$.
  \medskip

  To see that this filling operation is uniform, note that for an $f
  \co K \to I$ defined on $j,J$ and on $i$ the case which defines the
  filling of $[ J \mapsto \ws; (j,a) \mapsto w_{ja}] f$ coincides with
  the case used to defined $[ J \mapsto \ws; (j,a) \mapsto w_{ja}]$ by
  the injectivity requirement on $f$---uniformity then follows for
  each case separately since we only used operations that suitably
  commute with $f$ in the definition of the filling.  If $f$ is only
  defined on $j,J$ but not on $i$, the first case has to apply---to
  simplify notation assume $f$ is $\subst i c$---then by construction
  (equations~\eqref{eq:wlog-case-a} and~\eqref{eq:wlog-case-b})
  \begin{multline*}
    \bigl((\fGl\,\fstr_A\,\fstr_B\,\ext)\, (\rho,i)\,[ J \mapsto
    \ws; (j,a) \mapsto w_{ja}] \bigr) \subst i c %
    =\\
    (\fGl\,\fstr_A\,\fstr_B\,\ext)\, (\rho-i,c)\,[ J \mapsto \ws
    \subst i c; (j,a) \mapsto (w_{ja} \subst i c)],
  \end{multline*}
  concluding the proof.
\end{proof}



\begin{theorem}\label{thm:glue-kan-refined}
  We can refine the Kan structure $\fGl\,\fstr_A\,\fstr_B\,\ext$ given
  in Theorem~\ref{thm:glue-kan} such that it satisfies
  \begin{equation*}
    (\overline{\fGl\,\fstr_A\,\fstr_B\,\ext}) \, (\rho,i) \,[(i,0)
    \mapsto u] = t \rho \, u.
  \end{equation*}
\end{theorem}
\begin{proof}
  We modify the Kan structure given in the proof of
  Theorem~\ref{thm:glue-kan} to obtain the above equations.  The last
  two cases in the proof above where $i=j$ are modified by an
  additional case distinction on whether $J$ is empty or not.  If $J$
  is not empty or $a = 1$, proceed as before. In case $J$ is empty and
  $a = 0$, then we are given $u_{i0} \in A (\rho-i)$ and an empty tube
  and can define $(\fGl\,\fstr_A\,\fstr_B\,\ext)\,{(\rho,i)}\,[(i,0)
  \mapsto u_{i0}] = (u_{i0}, t \rho \,u_{i0})$.  That this definition
  remains uniform is proved as in Theorem~\ref{thm:glue-kan} using the
  observation that $\lvert J \rvert = \lvert J f \rvert$ for $f$
  defined on $J$. In addition we retain stability under substitution.
\end{proof}

\begin{remark}
  \label{rem:glue-kan-refined-inv}
  It is also possible to change the Kan structure such that it
  satisfies
  \[
    (\overline{\fGl\,\fstr_A\,\fstr_B\,\ext}) \, (\rho,i) \,[(i,1)
    \mapsto u] = t^{-1} \rho \, u,
  \]
  where $t^{-1}$ is the inverse of $t$ which can be constructed from
  $\ext$.  For this one also has to modify the case where $J$ is empty
  and $a = 1$ from the definition of $\fGl$ using $t^{-1}$ and that
  $t^{-1}$ is a (point-wise) right inverse of $t$ (in the sense of
  path types).  The latter is also definable using $\ext$.
\end{remark}

Let us recall the definition of a universe $\U$ of small Kan types
(assuming a Gro\-then\-dieck universe of small sets in the ambient set
theory).  A type $A \in \Ty(\Gamma)$ is small if all the sets $A \rho$
for $\rho \in \Gamma(I)$ are so.  A Kan type $(A,\fstr) \in
\KTy(\Gamma)$ is small if $A \in \Ty(\Gamma)$ is small.  We denote the
set of all such small types and Kan types by $\Ty_0(\Gamma)$ and
$\KTy_0(\Gamma)$, respectively.  Substitution makes both $\Ty_0$ and
$\KTy_0$ into presheaves on the category of cubical sets.  The
universe $\U$ is now given as $U = \KTy_0 \circ \yoneda$ where
$\yoneda$ denotes the Yoneda embedding. For an $I$-cube $(A,\fstr) \in
\U(I) = \KTy_0(\yoneda I)$ we have that $A$ is a presheaf on the
category of elements of $\yoneda I$, and $A(J,f)$ is a small set for
every element $(J,f\co J \to I)$.  Moreover, $\kappa(J,f)$ is a filler
function for open boxes in $A$ over $(J,f)$.  Of particular interest
are the small set $A(I,\id_I)$ and filler function $\kappa(I,\id_I)$.

Given $a \in \Ter(\Gamma,\U)$ we can associate a small type $\El a$ in
$\Ty_0 (\Gamma)$ by $(\El a) \rho = A(I,\id_I)$ where $a \rho =
(A,\fstr)$.  We equip $\El a$ with the Kan structure $\fEl\,a$ defined
by $(\fEl\,a)\rho = \kappa(I,\id_I)$.  This results in an isomorphism
which is natural in $\Gamma$:
\[
  \begin{tikzpicture}
    \node (Ter) at (0,0) {$\Ter(\Gamma,\U) $};
    \node (KTy) at (3,0) {$\KTy_0 (\Gamma),$};
    \draw[->,transform canvas={yshift=0.25em}] (Ter) --
    node [above]{\scriptsize$\langle{\El},{\fEl}\rangle$} (KTy);
    \draw[->,transform canvas={yshift=-0.25em}] (KTy) --
    node [below]{\scriptsize${\code{-}}$} (Ter);
  \end{tikzpicture}
\]
where $\code {X} \rho = X \hat\rho \in \U(I)$ for $X \in \KTy_0
(\Gamma)$. Here $\hat\rho \co \yoneda I \to \Gamma$ is the associated
substitution of $\rho \in \Gamma(I)$, that is, $\hat\rho f = \rho
f \in \Gamma(J)$ for any $f \co J \to I$.  Since moreover
$\Hom(\Gamma,\U) \cong \Ter (\Gamma,\U)$, we get that
$\KTy_0$ is representable.

\begin{theorem}
  \label{thm:univ-kan}
  $\U$ has a Kan structure.
\end{theorem}
\begin{proof}
  \cite[Theorem~4.2]{Huber15}.
\end{proof}

We are now ready for the first main result of this paper.

\begin{theorem}[Univalence]
  \label{thm:ua-path}
  The type
  \[
    \Pi(a : \U)\, %
    \isContr \bigl(\Sigma (b : \U) \, \Equiv\,(\El a)\,(\El b) \bigr)
  \]
  in $\Ty (\emptyctx)$ has a section, where $\emptyctx$ denotes the empty context.
\end{theorem}
\begin{proof}
  Because our operation $\Gl$ preserves smallness we obtain an operation
  turning an equivalence between small Kan types into a path in $\U$:
  given $a \in \Ter(\Gamma, \U)$ and $b \in \Ter(\Gamma, \U)$ with $t
  \in \Ter(\Gamma, \Equiv\,(\El a)\,(\El\,b))$ we get a small type $\Gl
  (t.1) \in \Ty_0 (\Gamma * \II)$ which has the Kan structure $\fstr =
  \fGl\,(\fEl a)\,(\fEl b)\,\ext$ by Theorem~\ref{thm:glue-kan} where
  $\ext$ is constructed from $\fEl a$, $\fEl b$, and $t$ using
  Lemma~\ref{lem:contractible}.  Hence $\code{ (\Gl (t.1), \fstr)} \in
  \Ter(\Gamma *\II, \U)$ with $\code {(\Gl (t.1),\fstr)} [0] = \code
  {(\Gl (t.1) [0], \fstr [0])} = \code {(\El a,\fEl a)} = a$ and
  likewise $\code {(\Gl (t.1),\fstr)} [1] = b$.  Finally, abstracting
  gives a path $\nabs {} \code {(\Gl (t.1), \fstr)} \in \Ter(\Gamma,
  \Pth{\U}\,a\,b)$.

  Choosing $a : \U, b : \U, t : \Equiv\,(\El a)\,(\El b)$ as the context
  $\Gamma$ above we get using currying
  \[
    \uaetop \in \Ter\bigl(\emptyctx ,\Pi (a\,b : \U) (\Equiv\,(\El
    a)\,(\El b) \to \Pth{\U}\,a\,b)\bigr).
  \]
  Observe that we didn't use that $\Gl$ and its Kan structure commute
  with substitutions to derive $\uaetop$.

  In addition to $\uaetop$ we obtain a section $\uabeta$ of
  \[
    \Pi(a\,b : \U) \,\Pi(t : \Equiv\,(\El a)\,(\El b)) ~ \Pth{\El a
      \to \El b}\,(\transp\,(\uaetop\,t))\,(t.1)
  \]
  where $\transp : \Pth{\U}\,a\,b \to \El a \to \El b$ is the
  transport operation for paths for the type $(\El \qq,\fEl \qq) \in
  \KTy_0 (\U)$ (see the operation $\mathsf{T}$
  in~\cite[Section~8.2]{BezemCoquandHuber14}).  Indeed, the path to
  justify $\uabeta$ is given by reflexivity using our refined Kan
  structure from Theorem~\ref{thm:glue-kan-refined} plus that
  $\transp$ is given in terms of composition with an empty tube.

  The transport operation $\transp$ can easily be extended to an
  operation
  \[
   \trpeqv : \Pth{\U}\,a\,b \to  \Equiv\,(\El a)\,(\El b)
  \]
  which goes in the opposite direction as $\uaetop$.  Actually,
  $\uaetop$ and $\trpeqv$ constitute a section-retraction pair because
  of $\uabeta$ and the fact that $\isEquiv t.1$ is a proposition, that
  is, all its inhabitants are path-equal.  Hence also $\Sigma (b : \U)
  \, \Equiv\,(\El a)\,(\El b)$ is a retract of $\Sigma (b : \U) \,
  \Pth{\U}\,a\,b$.  Since $\U$ has a Kan structure by
  Theorem~\ref{thm:univ-kan}, the latter type is contractible (see
  \cite[Section~8.2]{BezemCoquandHuber14}) and thus so is the former,
  concluding the proof.
\end{proof}



\section{Identity types}\label{sec:Id}

We will now describe the identity type which justifies the usual
judgmental equality for its eliminator following Swan~\cite{Swan14}.

Let $\Gamma$ be a cubical set and $A,B\in \Ty(\Gamma)$, i.e., $A$ and
$B$ are presheaves on the category of elements of $\Gamma$.  For
natural transformations\footnote{Natural transformations $\alpha \co A
  \to B$ correspond to sections in $\Ter (\Gamma, A \to B)$, and also
  to maps between the projections $\Gamma.A \to \Gamma$ and $\Gamma.B
  \to \Gamma$ in the slice over $\Gamma$.  To simplify notation, we
  will write $\alpha$ for either of these.}  $\alpha \co A \to B$ we
are going to define a \emph{factorization} as $\alpha = p_\alpha \,
i_\alpha$ with $i_\alpha \co A \to M_\alpha$ and $p_\alpha \co
M_\alpha \to B$.  Furthermore, $i_{\alpha}$ will be a
\emph{cofibration} (i.e., has the lifting property w.r.t.\ any
acyclic fibration as formulated in Corollary~\ref{cor:equiv-to-afib})
and $p_{\alpha}$ will be equipped with an \emph{acyclic-fibration}
structure.  This factorization corresponds to Garner's factorization
using the refined small object argument~\cite{Garner09} specialized to
cubical sets.

For $\rho$ in $\Gamma(I)$ we will define the sets $M_\alpha \rho$
together with the restriction maps $M_\alpha \rho \to M_\alpha (\rho
f)$ (for $f \co J \to I$) and the components $M_\alpha \rho \to B
\rho$ of the natural transformation $p_\alpha$ by an inductive process
(see Remark~\ref{rem:m-alpha-construction} below).  The elements of
$M_\alpha \rho$ are either of the form $\i \, u$ with $u$ in $A \rho$
(and $\i$ considered as a constructor) and we set in this case $(\i \,
u) f = \i (u \, f)$ and $p_\alpha (\i \, u) = \alpha\,u$.  Or
the elements are of the form $(v, [J \mapsto \us])$ where $v \in
B\rho$, $J \subseteq I$, and $\us$ is a $J$-tube in $M_\alpha \rho$
over $v$ (meaning $p_\alpha \, u_{jb} = v \subst j b$).  In the latter
case we set $p_\alpha (v, [J \mapsto \us]) = v$ and for the
restrictions $(v,[J\mapsto \us])f = u_{jb}(f-j)$ if
$f(j)=b\in\set{0,1}$ for some $j\in J$, and $(v,[J\mapsto \us])f =
(vf,[Jf\mapsto\us f])$ if $f$ is defined on $J$.  Note that
restrictions do not increase the syntactic complexity of an element $m
\in M_\alpha \rho$.  This defines $M_\alpha \in \Ty (\Gamma)$ and we
set $i_\alpha\,u = \i\,u$.

\begin{remark}
  \label{rem:m-alpha-construction}
  This construction is rather subtle in a set-theoretic framework.
  One possible way to define this factorization is to first
  inductively define larger sets $M'_\alpha \rho$ containing all
  formal elements $\i\,u$ with $u \in A \rho$, and $(v, [J \mapsto
  \us])$ with $v \in B\rho$ and where $\us$ is represented by a family
  of elements $u_f \in M'_\alpha (\rho f)$ indexed by all $f\co K \to
  I$ with $f j = 0$ or $1$ for some $j \in J$, but without requiring
  compatibility.  On these sets one can then define maps $M'_\alpha
  \rho \to M'_\alpha (\rho f)$ and $M'_\alpha \rho \to B \rho$.  Given
  these maps, we can single out the sets $M_\alpha \rho \subseteq
  M'_\alpha \rho$ of the well-formed elements as in the definition
  above, on which the corresponding maps then induce restriction
  operations (satisfying the required equations) and the natural
  transformation $p_\alpha$.
\end{remark}

We use $M_\alpha, i_\alpha, p_\alpha$ in the following way. Let $A$ be
a Kan type and let $B=\Pth{A}$ be the Kan type of paths over $A$
without specified endpoints. (The Kan structure on $A$ induces the Kan
structure on $B$, much in the same way as shown
in~\cite{BezemCoquandHuber14} for types $\Pth{A}\,a\,b$.)  As
mentioned in Section~\ref{sec:paths}, transport along reflexivity
paths is not necessarily the identity function.  One could solve this
problem if one could recognize the reflexivity paths, which is not
possible in $\Pth{A}$.  Swan's~\cite{Swan14} solution to this problem
is to define a type equivalent to $\Pth{A}$ in which one can recognize
(representations of) reflexivity paths.  This is the type $M_\alpha$
with $\alpha \co A \to \Pth{A}$ mapping each $a$ in $A$ to its
reflexivity path.  The representation of the reflexivity path of $a$
in $M_\alpha$ is $\i\,a$, with $\i$ a constructor of the inductively
defined type $M_\alpha$, and recognizing $\i\,a$ is done through
pattern matching.  All the rest of the complicated definition above is
to make sure that $M_\alpha$ has the right Kan structure
(Lemma~\ref{lem:MalphaKan}), and that elimination generally has the
right properties (Corollary~\ref{cor:ialpha-cofib}).


\medskip
Constructors of the form $(v,[J\mapsto\us])$ equip $p_\alpha \co
M_\alpha \to B$ with an acyclic-fibration structure which (uniformly)
fills tubes $[J \mapsto \us]$ in $M_\alpha \rho$ over a filled cube
$v$ in $B \rho$.  Thus to, say, construct a path between specified
endpoints in $M_\alpha$ it is enough to give a path in $B$ between the
images of the endpoints under $p_\alpha$.

There are two important observations to make at this point: First,
this construction preserves smallness, i.e., $M_\alpha \in \Ty_0
(\Gamma)$ whenever $A, B \in \Ty_0 (\Gamma)$.  And, second, this
construction is stable under substitution: given $\sigma \co \Delta
\to \Gamma$ we have $M_\alpha \sigma = M_{\alpha \sigma}$, $i_\alpha
\sigma = i_{\alpha \sigma}$, and $p_\alpha \sigma = p_{\alpha
  \sigma}$.  Neither of these properties holds for the corresponding
factorization into an acyclic cofibration followed by a fibration
(sketched in~\cite[Section~3.5]{Huber15} for a special case).

\begin{lemma}
  \label{lem:MalphaKan}
  Given $\fstr \in \Fill (\Gamma,B)$ there is $\underline{M}_\alpha
  \fstr \in \Fill(\Gamma,M_{\alpha})$.  Moreover, this assignment is
  stable under substitution, i.e., $(\underline{M}_\alpha \fstr)
  \sigma = \underline{M}_{\alpha \sigma}(\fstr \sigma)$ for $\sigma
  \co \Delta \to \Gamma$.
\end{lemma}
\begin{proof}
  Let $\rho \in \Gamma(I)$ and, say, $m = [ J \mapsto \ms; (i,0)
  \mapsto m_{i0}]$ be an open box in $M_\alpha$ over $\rho$.  We get
  an open box $v = [J \mapsto \vs; (i,0) \mapsto v_{i0}]$ in $B$ over
  $\rho$ by setting $v_{jb} = p_\alpha\,m_{jb}$.  We define
  \[
    \underline{M}_\alpha\,\fstr\,\rho\,m = %
    (\fstr\,\rho\,v, [J,i\mapsto \ms,m_{i0},m_{i1}])
  \]
  with $m_{i1} = ({\bar\fstr}\,\rho\,v, [J \mapsto \ms \subst i 1])$.
\end{proof}

\begin{lemma}
  \label{lem:lift-via-path}
  Given $(D,\fstr_D) \in \KTy (\Gamma.M_\alpha)$ and sections $s \in
  \Ter (\Gamma.A, D i_\alpha)$ and $s' \in \Ter (\Gamma.M_\alpha, D)$
  together with a homotopy
  \[
    e \in \Ter \bigl(\Gamma, \Pi(a : A)\,\Pth{D (i_\alpha\, a)}\,(s'
    (i_\alpha \,a))\,(s\,a)\bigr),
  \]
  it is possible to find a section $\tilde s \in \Ter
  (\Gamma.M_\alpha, D)$ such that $\tilde s i_\alpha = s \in
  \Ter(\Gamma.A,D i_\alpha)$.  Or stated as a diagram,
  we are given a commuting square
  \[
    \begin{tikzcd}
      \Gamma.A \rar[""{name=Eend, swap}] {(i_\alpha,s)} \dar
      [swap] {i_\alpha} & \Gamma.M_\alpha.D  
      \dar
      \\
      \Gamma.M_\alpha \rar[equal] \urar [""{name=Estart, near start },
      "s'" description] \uar[from=Estart, to=Eend, Rightarrow, bend
      right=10, near start, "e"] & \Gamma.M_\alpha
    \end{tikzcd}
  \]
  where the upper left triangle only commutes up to the homotopy $e$
  and the lower triangle commutes strictly, and we get a new diagonal
  lift where both triangles commute strictly.  Moreover, this
  assignment is stable under substitutions, i.e., given $\sigma \co
  \Delta \to \Gamma$, substituting the chosen diagonal lift $\tilde s$
  (for the data $\alpha,D,\fstr_D,s,s',e$) along $(\sigma\pp,\qq) \co
  \Delta.M_{\alpha \sigma} \to \Gamma.M_\alpha$ results in the chosen
  diagonal lift for the substituted data (where $\sigma$ is weakened
  appropriately if needed).
\end{lemma}


\begin{proof}
  For $\rho \in \Gamma(I)$ and $m \in M_\alpha \rho$ we define $\tilde
  s (\rho,m) \in D (\rho,m)$ and a path $\tilde e (\rho,m)$ between
  $s' (\rho,m)$ and $\tilde s (\rho,m)$ in $D (\rho,m)$ by induction
  on the syntactic complexity of $m \in M_\alpha \rho$ such that
  $(\tilde s(\rho,m))f = \tilde s (\rho f, m f)$ and $(\tilde e
  (\rho,m)) f = \tilde e (\rho f,m f)$.  In case $m = \i\,u$ for $u
  \in A \rho$, we set $\tilde s (\rho, \i\,u) = s (\rho, u)$ and
  $\tilde e (\rho, \i\,u) = e (\rho, u)$.  In case $m = (v, [J \mapsto
  \ms])$, we set
  \[
    \tilde e (\rho, m) = %
    \nabs i \, \fstr_D\,(\rho \deg_i)\,[J \mapsto \ws;(i,0) \mapsto s'
    (\rho,m) ]
  \]
  where $w_{jb} = \tilde e (\rho \subst j b, m_{jb}) \napp i$, and
  correspondingly $\tilde s (\rho, m) = \tilde e (\rho, m) \napp 1$.
\end{proof}

If the Kan structure is an acyclic-fibration structure as in
Definition~\ref{def:acyclic-fibration}, that is, if we can fill tubes
without a closing lid, the above proof can be carried out without
$s'$.  This implies the following result, which expresses that
$i_{\alpha} \co A \to M_{\alpha}$ is a {\em cofibration}.

\begin{corollary}
  \label{cor:ialpha-cofib}
  Given $D \in \Ty (\Gamma.M_\alpha)$ with an acyclic-fibration
  structure and a section $s \in \Ter (\Gamma.A, D i_\alpha)$ it is
  possible to define a section $\tilde s \in \Ter (\Gamma.M_\alpha,
  D)$ such that $\tilde s i_\alpha = s \in \Ter (\Gamma.A, D
  i_\alpha)$.  That is, there is a diagonal lift in the diagram:
  \[
    \begin{tikzcd}
      \Gamma.A \rar{(i_\alpha,s)} \dar [swap] {i_\alpha} &
      \Gamma.M_\alpha.D
      \dar
      \\
      \Gamma.M_\alpha \rar[equal] \urar [dashed]{\tilde s}&
      \Gamma.M_\alpha
    \end{tikzcd}
  \]
  Moreover, this assignment is stable under substitution.
\end{corollary}
\begin{proof}
  By Lemma~\ref{lem:contractible} we know that $D$ has a Kan structure
  and is contractible.  From the contractibility we get a section $s'
  \in \Ter (\Gamma.M_\alpha,D)$ and a homotopy between $s' \,
  i_\alpha$ and $s$, and can thus apply Lemma~\ref{lem:lift-via-path}
  to get a strict diagonal filler.
\end{proof}

This also implies the following result, which expresses that
$i_{\alpha} \co A \to M_{\alpha}$ is a \emph{acyclic} cofibration as
soon as $\alpha$ has a well-behaved homotopy inverse.  Recall that
application $\ap\,\alpha\,p \in \Ter(\Gamma, \Pth{B}\,
(\alpha\,u)\,(\alpha\,v))$ of $\alpha\co A \to B$ to a path $p \in
\Ter(\Gamma,\Pth{A}\,u\,v)$ is given by $(\ap\,\alpha\,p)\rho = \nabs
i \, \alpha (p \rho \napp i)$ (see~\cite[Section~3.3.2]{Huber15}).

\begin{corollary}
  \label{cor:equiv-to-afib}
  Let $\alpha \co A \to B$ and assume we are given $\beta : B \to A$
  and sections
  \begin{align*}
    \eta
    &\in \Ter \bigl(\Gamma, \Pi (a:A)\, \Pth{A}\,(\beta (\alpha \,
      a))\,a\bigr),
    \\
    \epsilon
    &\in \Ter \bigl(\Gamma, \Pi (b : B)\,\Pth{B}\,(\alpha (\beta \,
      b))\,b\bigr), \text{ and}
    \\
    \tau
    &\in \Ter \bigl(\Gamma, \Pi (a: A)\, \Pth{}\,(\epsilon (\alpha \,
      a))\,(\ap\,\alpha\,(\eta\,a))\bigr),
  \end{align*}
  where the omitted subscript of the path-type in $\tau$ is
  $\Pth{B}\,(\alpha\,(\beta (\alpha \, a)))\,(\alpha\,a)$.  Then given
  $D \in \Ty (\Gamma.M_\alpha)$ with Kan structure $\fstr_D$ we can
  extend any section $s \in \Ter (\Gamma.A, D i_\alpha)$ to a section
  $\tilde{s} \in \Ter (\Gamma.M_\alpha, D)$ satisfying $\tilde s \,
  i_\alpha = s$.  Moreover, this assignment is stable under
  substitution.
\end{corollary}
\begin{proof}
  It is sufficient to construct $s'$ and $e$ as in
  Lemma~\ref{lem:lift-via-path}.  To enhance readability we omit the
  arguments from $\Gamma$.

  First, given $m \in M_\alpha$ we have a path $m^*$ connecting
  $i_\alpha (\beta (p_\alpha m))$ to $m$, since the images of the
  endpoints under $p_\alpha$ are $\alpha (\beta (p_\alpha \, m))$ and
  $p_\alpha\,m$ which are connected by $\epsilon (p_\alpha\,m)$.  Thus
  the acyclic-fibration structure on $p_\alpha$ gives us a desired
  path $m^*$, which moreover lies over $\epsilon (p_\alpha\,m)$, i.e.,
  \begin{equation}
    \label{eq:lift-over-eps}
    p_\alpha (m^* \napp j) = \epsilon (p_\alpha\,m) \napp j \quad
    \text{for fresh $j$.}
  \end{equation}

  Next, we have $s (\beta (p_\alpha \, m)) \in D (i_\alpha (\beta
  (p_\alpha \, m)))$ which we then can transport to $D \, m$ using the
  Kan structure and the path $m^*$.  Thus we set
  \begin{equation}
    \label{eq:total-section}
    s' m := \bar\fstr_D \,(m^* \napp j)\,[(j,0) \mapsto s (\beta
    (p_\alpha \, m))].
  \end{equation}

  It remains to give a path $e\,a$ connecting $s' (i_\alpha \,a)$ to
  $s\,a$ in $D (i_\alpha \, a)$ for $a \in A$.  We have the two
  horizontal lines (in direction $j$) in
  \begin{equation}
    \label{eq:homotopy-problem}
    \begin{tikzcd}[column sep=2cm]
      s (\beta (\alpha \, a)) \rar
      & s \, a
      \\
      s (\beta (\alpha \, a))%
      \rar %
      \uar [equal]
      & s' (i_\alpha \, a) %
      \uar [dashed,swap] {e \, a}
    \end{tikzcd}
  \end{equation}
  where the top line is given by $s (\eta \, a \napp j)$ in $D
  (i_\alpha (\eta\,a \napp j))$ and the bottom line is given by a
  filling in $D ((i_\alpha \, a)^* \napp j)$ following the
  construction~\eqref{eq:total-section} of $s'$.  We want to construct
  the vertical dashed line in $D (i_\alpha \, a)$.  We can define this
  line using a composition on the open box specified
  in~\eqref{eq:homotopy-problem} as soon as we can provide an interior
  of the following square in $M_\alpha$ over
  which~\eqref{eq:homotopy-problem} is an open box:
  \begin{equation*}
    \begin{tikzcd}[column sep=2cm]
      i_\alpha (\beta (\alpha \, a)) %
      \rar {i_\alpha (\eta\,a \napp j)} %
      & i_\alpha \, a
      \\
      i_\alpha (\beta (\alpha \, a)) %
      \rar [swap] {(i_\alpha \, a)^* \napp j}
      \uar [equal]
      & i_\alpha \, a
      \uar [equal]
    \end{tikzcd}
  \end{equation*}
  But by~\eqref{eq:lift-over-eps}, mapping this square to $B$ using
  $p_\alpha$ has a filler given by $\tau\,a\napp k \napp j$ (where $k$
  extends vertically), and thus also a filler in $M_\alpha$ since
  $p_\alpha$ has an acyclic-fibration structure, concluding the proof.
\end{proof}

The representation of the identity type with the usual judgmental
equality for its eliminator follows from these results by considering
the case where $B$ is the type of paths without specified endpoints
$\Pth{A}$ over a type $A$ and $\alpha\,a$ is the constant path~$a$.
We get a factorization with $\Id{A} := M_\alpha$, $\refl := i_\alpha$,
and where the right vertical map is given by taking endpoints:
\begin{equation*}
  \begin{tikzcd}
    \Id{A} \rar 
    &
    \Path{A} \dar 
    &
    \\
    A \arrow{r} [swap] {\Delta} %
    \uar {\refl}%
    \urar {\alpha} %
    &
    A \times A
  \end{tikzcd}
\end{equation*}
This $\alpha$ satisfies the hypothesis of
Corollary~\ref{cor:equiv-to-afib} using the properties of path-types
from~\cite[Section~8.2]{BezemCoquandHuber14}, and hence $\refl \co A
\to \Id{A}$ has diagonal lifts against types with Kan structure. These
diagonal lifts serve as the interpretation of the eliminator (cf.\
\cite[p.52]{AwodeyWarren09}) and their choice is stable under
substitution, allowing us thus to interpret identity types.

One can also explain $\Id{A}$ with fixed endpoints as Kan type in
context $\Gamma.A.A\pp$ and then show that $\Id{A}\,u\,v$ is
$\Path$-equivalent to $\Pth{A}\,u\,v$.  It follows that a type is
$\Pth{}$-contractible if, and only if, it is $\Id{}$-contractible.
The univalence axiom for $\Pth{}$-types (Theorem~\ref{thm:ua-path})
hence also holds formulated with $\Id{}$-types.

We can summarize the results of this section as:
\begin{theorem}
  \label{thm:main}
  The cubical set model of~\cite{BezemCoquandHuber14,Huber15} supports
  identity types and validates the univalence axiom.
\end{theorem}

\begin{acknowledgements*}
  We want to thank Cyril Cohen and Anders Mörtberg for several
  discussions around an implementation of this system, as well as
  Andrew Swan for discussions on the representation of the identity
  type in the cubical set model.
\end{acknowledgements*}

\bibliographystyle{amsplain}
\bibliography{csetua}

\end{document}